\documentclass{amsart}
\usepackage{amssymb,graphicx,rotating,multirow,multicol}
\usepackage{amsmath}
\usepackage{amsthm}
\usepackage{amsfonts}
\usepackage{pdfsync}
\usepackage{array}
\usepackage{mathrsfs}
\usepackage[toc,page]{appendix}
\usepackage{caption}
\usepackage{subfigure}
\usepackage{epstopdf}
\usepackage{tikz-cd}

\swapnumbers
\newtheorem{theorem}{Theorem}[subsection]

\newtheorem{lemma}[theorem]{Lemma}

\newtheorem{proposition}[theorem]{Proposition}
\theoremstyle{remark}
\newtheorem{remark}[theorem]{Remark}
\numberwithin{equation}{subsection}
\newcommand{\Z}{\ensuremath{\mathbb{Z}}}

{\catcode`\@=11 \gdef\mnote#1{\marginpar{\footnotesize
 \tolerance\@M\spaceskip2.6\p@ plus10\p@ minus.9\p@\rm#1}}}
\let\Bbb\mathbb
\let\Cal\mathcal
\def\D{\Delta}
\def\G{\Gamma}
\def\P{\Cal P}
\def\L{\Cal L}
\def\s{\sigma}
\def\CL{\mathbf{L}}
\def\CM{\mathbf{M}}
\def\sm{\smallsetminus}
\def\lk{\operatorname{lk}}
\let\le\leqslant 
\let\la\langle
\let\ra\rangle

\def\Rp#1{\Bbb{RP}^{#1}}
\def\CCC{\Cal C}
\def\Sch{\operatorname{\mathbb{S}ch}}
\def\sign{\operatorname{sgn}}
\makeatletter
\newcommand{\addresseshere}{%
  \enddoc@text\let\enddoc@text\relax
}
\makeatother
\begin{document}
\par
\renewcommand{\arraystretch}{1.2}
\title[Real Schl\"afli Six-Line Configurations]{Topology of Real Schl\"afli Six-Line Configurations on Cubic Surfaces and in $\Rp3$}
\author[S.~Finashin, R.A.~Zabun]{Sergey Finashin, Remz\.{i}ye Arzu Zabun}
\address{Department of Mathematics, Middle East Tech. University\\
 06800 Ankara Turkey}
\address{Department of Mathematics, Gaziantep University\\
 27310 Gaziantep Turkey}


\date{}
\dedicatory{}
\begin{abstract}
A famous configuration of 27 lines on a non-singular cubic surface in $\mathbb P^3$
contains remarkable subconfigurations, and in particular the ones formed by six pairwise disjoint lines.
We study such six-line configurations in the case of real cubic surfaces from topological viewpoint, as configurations of six disjoint lines
in the real projective 3-space, and show that the condition that they
lie on a cubic surface implies a very special property of {\it homogeneity}. This property distinguish them
in the list of 11 deformation types of configurations formed by six disjoint lines in $\mathbb{RP}^3$.
\end{abstract}

\maketitle

\hbox{\hskip30mm\rightline{\vbox{\hsize90mm
\noindent\baselineskip10pt{\it\footnotesize
Drawing at pleasure five lines $a,b,c,d,e$ which meet a line $F$, then may any four of the five lines be intersected by another line besides $F$.
Suppose that $A,B,C,D,E$ are the other lines
intersecting $(b,c,d,e)$, $(c,d,e,a)$, $(d,e,a,b)$, $(e,a,b,c)$ and $(a,b,c,d)$ respectively. Then ...
there will be a line $f$ intersecting the five lines $A$, $B$, $C$, $D$, $E$.''
\vskip3mm
\noindent
\rm
\vbox{\noindent
{\bf Ludwig Schl\"afli}
\textit{``An attempt to determine the twenty-seven lines upon a
surface of the third order, and to divide such surfaces into species in
reference to the reality of the lines upon the surface''}
}}}}}

\vskip3mm
\section{Introduction}\label{introduction}
\subsection{Segre classification of Real Schl\"afli (double) sixes}\label{Segre-classification}
A sextuple of skew lines, $\L=\{L_1,\dots,L_6\}$, in $\mathbb{P}^3$ will be called
a {\it Schl\"afli six} if there is some non-singular cubic surface, $X\subset \mathbb{P}^3$, containing all these lines.
A necessary and sufficient condition for existence of such cubic was found by Schl\"afli \cite{Sch}: each subset of 5 lines $\L\sm\{L_i\}$ must have a common transversal,
$L_i'$, but all 6 lines of $\L$ should not have a common transversal.
Then configuration $\L'=\{L_1',\dots,L_6'\}$ is also a Schl\"afli six called  {\it complementary to $\L$}.
Together $\L$ and $\L'$ form a {\it Schl\"afli double six} which is traditionally written in the form
$\binom{L_1\,L_2\,L_3\,L_4\,L_5\,L_6}{L_1'\,L_2'\,L_3'\,L_4'\,L_5'\,L_6'}$.

Note that cubic $X$ is determined uniquely by a Schl\"afli six $\L$, since $\L$ determines $\L'$, and
two irreducible cubics cannot have 12 common lines by Bezout theorem.

Schl\"afli found that cubic surfaces may have 5 real forms, $F_1$,\dots, $F_5$ that contain respectively 27, 15, 7, 3, and 3 real lines (among 27 complex ones).
We will be concerned only about {\it real Schl\"afli six} that is by definition Schl\"afli six whose lines are real.
As Schl\"afli observed, it is only cubics of type $F_1$ that may contain real Schl\"afli sixes. In particular, for any real Schl\"afli six $\L$
the complementary one, $\L'$ is also real, and they form together a {\it real Schl\"afli double six}.

A deformation classification
of real Schl\"afli (double) sixes was performed by B.~Segre \cite{Se}, \textsection{61}, who found 4 kinds of them.
One way to distinguish these 4 kinds is to blow down the lines of a real Schl\"afli six on a cubic $X$ which gives a real projective plane with 6 marked points in general position;
``generality'' here means that no triple of points is collinear and all six are not coconic.
Figure \ref{adjgraph6} shows 4 six-point configurations on $\Rp2$ representing these 4 kinds of real Schl\"afli six.
We enhanced these configurations $\P$ to {\it inseparability graphs}.
\begin{figure}[h]
\subfigure[I.Icosahedral]{\scalebox{0.42}{\includegraphics{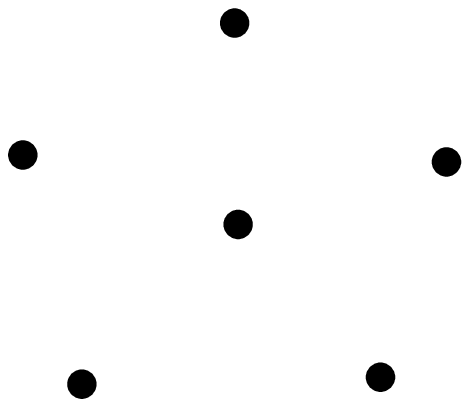}}}\hspace{0.8cm}
\subfigure[II.Bipartite]{\scalebox{0.42}{\includegraphics{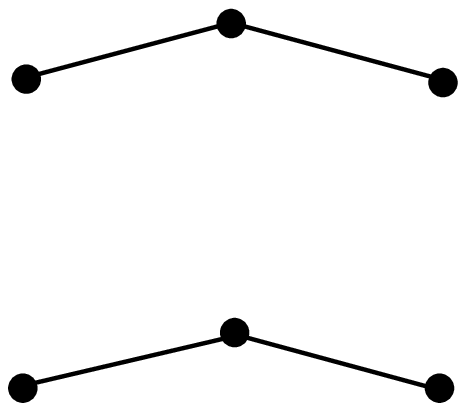}}}\hspace{0.8cm}
\subfigure[III.Tripartite]{\scalebox{0.42}{\includegraphics{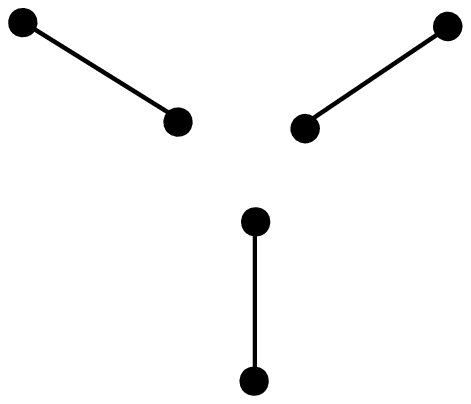}}}\hspace{0.8cm}
\subfigure[IV.Hexagonal]{\scalebox{0.42}{\includegraphics{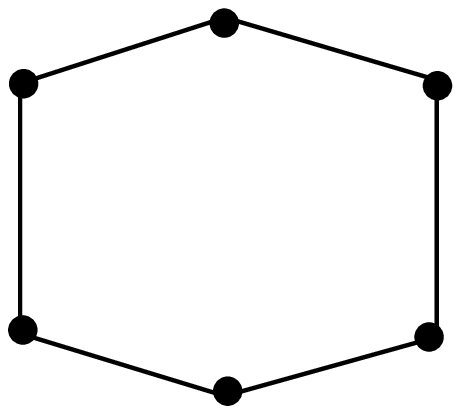}}}
\caption{4 kinds of 6-point configurations with inseparability graphs}
  \label{adjgraph6}
\vskip-3mm\end{figure}
Namely, we join
a pair of points of a configuration by an edge (line segment in $\Rp2$)
if and only if it has no intersections with the lines connecting pairwise the remaining four points.
According to the shape of this graph, we named these 6-point configurations
{\it icosahedral, bipartite, tripartite}, and {\it hexagonal} (see Figure \ref{adjgraph6})
 and use the same names also for the corresponding Schl\"afli sixes.
 Segre observed that the complementary Schl\"afli six $\L'$ has to be of the same kind as $\L$.

To each of the four kinds of Schl\"afli sixes $\L$, Segre (see \cite{Se}, \textsection{56}) associated
a diagram that we call {\it Segre pentagram}.
Namely,  let us denote by $\pi_i$, $i=1,\dots 5$ the plane spanned by lines $L_6\in \L$ and $L_i'\in\L'$ and numerate
the lines in our double six so that planes $\pi_1,\dots,\pi_5$ appear consecutively in the pencil of planes passing through $L_6$.
 Then the five intersection points $s_i=L_6\cap L_i'$, $i=1,\dots 5$ follow on line $L_6$
 in some (cyclic) order $s_{\s(1)},\dots,s_{\s(5)}$, where $\s\in S_5$ is a permutation.
 The Segre pentagram of $\L$ is obtained from a regular pentagon whose vertices are cyclically numerated by $1,\dots,5$ and connected by a broken line
$\s(1)\dots\s(5)\s(1)$.
\begin{figure}[h]
\centering
\scalebox{0.35}{\includegraphics{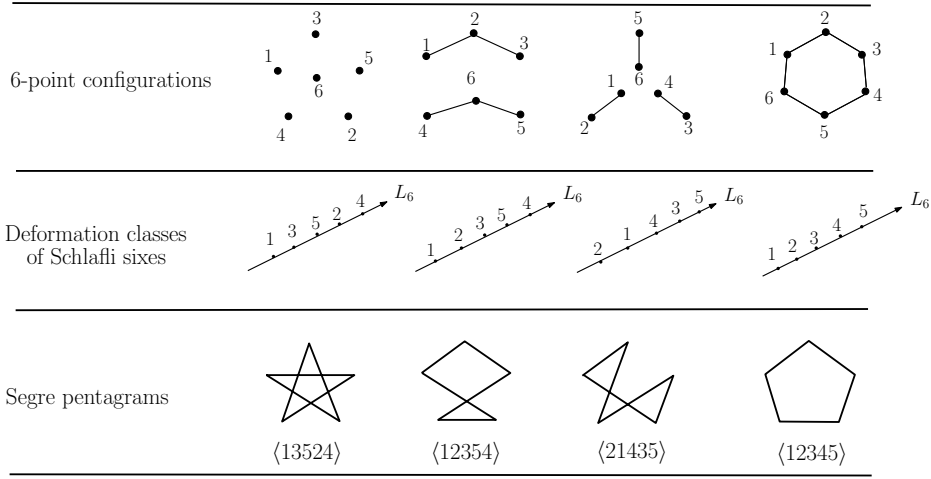}}
\caption{Segre pentagrams of real Schl\"afli sixes}
\label{ACC6}
\vskip-2mm\end{figure}
In terms of 6-point configurations, this corresponds to labeling points by $1,\dots,5$ consecutively with respect to the pencil of lines passing through point $6$;
then  on the conic passing through these  5 points, their consecutive order is $\s(1),\dots,\s(5)$
(see Figure \ref{ACC6}).

\subsection{The main result}\label{results}
Our aim in this paper is to find the place of the four kinds of real Schl\"afli sixes
in the Mazurovski\u\i~ deformation classification of configurations of 6 skew lines in $\Rp3$.
By {\it deformation} in that context we mean a path is the space of all configurations of 6 skew lines.
The Mazurovsk\u\i~ classification is expressed in terms of {\it join $n$-configurations} formed by $n$ skew lines that have two common transversals.
Namely, let us fix a pair of skew lines $L^p,L^q\subset\Rp3$ oriented so that the linking number $\lk(L^p,L^q)$ is negative,
and choose points $p_1,\dots, p_n\in L^p$ and $q_1,\dots,q_n\in L^q$ numerated consecutively with respect to the orientations.
Then with any permutation $\s\in S_n$ we can associate join configuration $J(\s)=\{L_1,\dots,L_n\}$, where line $L_i$ connects points $p_i$ and
$q_{\s(i)}$.
The choice of other lines $L^p$, $L^q$, or other points $p_i$, $q_i$, clearly, does not change $J(\s)$ up to deformation.
By contrary, if points $p_i$ or $q_i$ are numerated in the opposite order
(in accord with orientations of $L^p$ and $L^q$ that makes them negatively linked)
a similar construction yields instead of $J(\s)$ its {\it mirror partner} $\bar J(\s)$, defined as the image of $J(\s)$
under an orientation reversing projective transformation $\Rp3\to\Rp3$.

Mazurovski\u\i~, using calculation of a version of Jones polynomial,
 observed that configurations $\mathbf{M}$ and $\mathbf{L}$ on Figure \ref{ML} cannot be deformed into
join configurations.
\begin{figure}[h]
\vskip-2mm\centering
\subfigure[$J(123654)$]{\scalebox{0.28}{\includegraphics{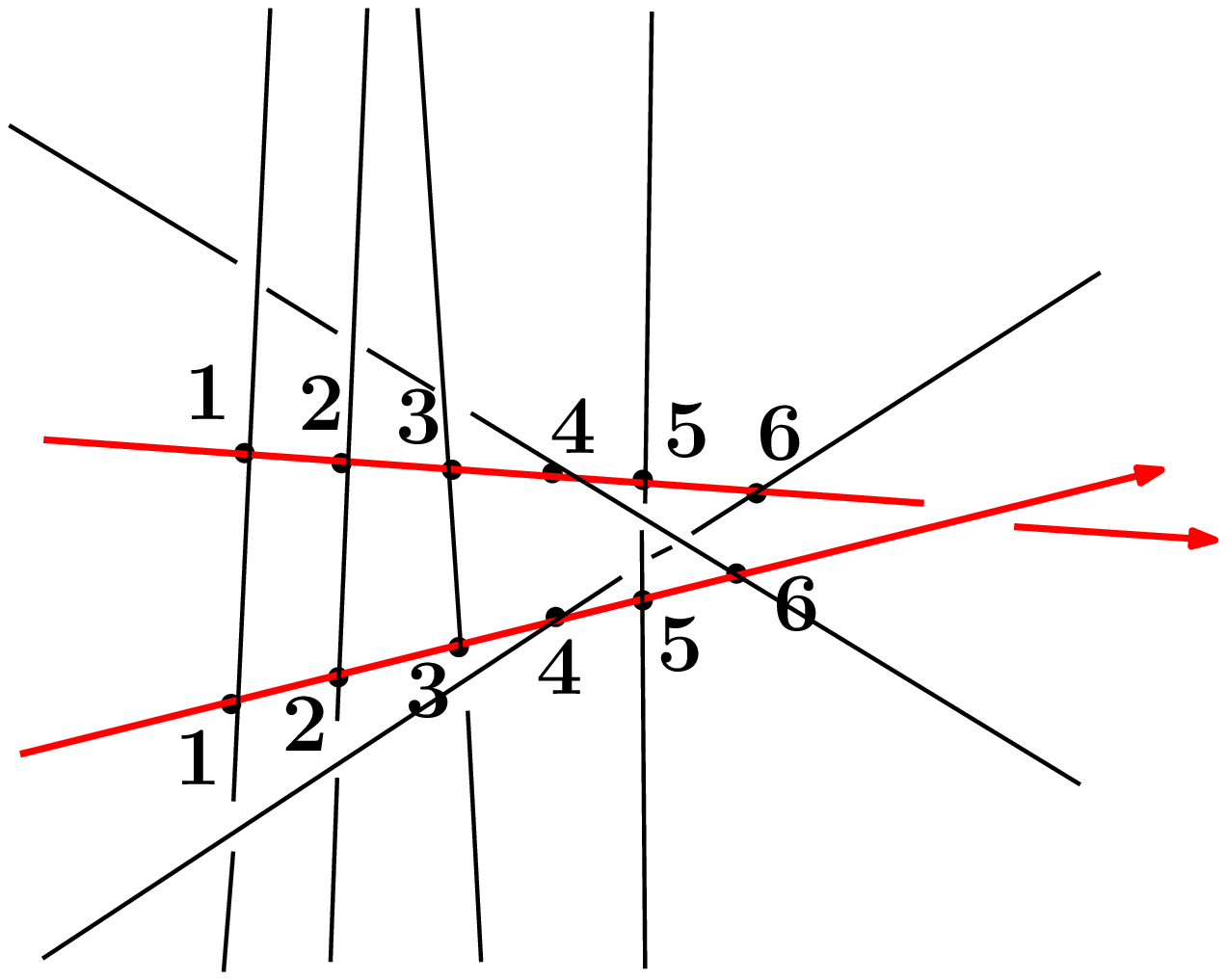}}}\hspace{0.4mm}
\subfigure[$\mathbf{M}$]{\scalebox{0.24}{\includegraphics{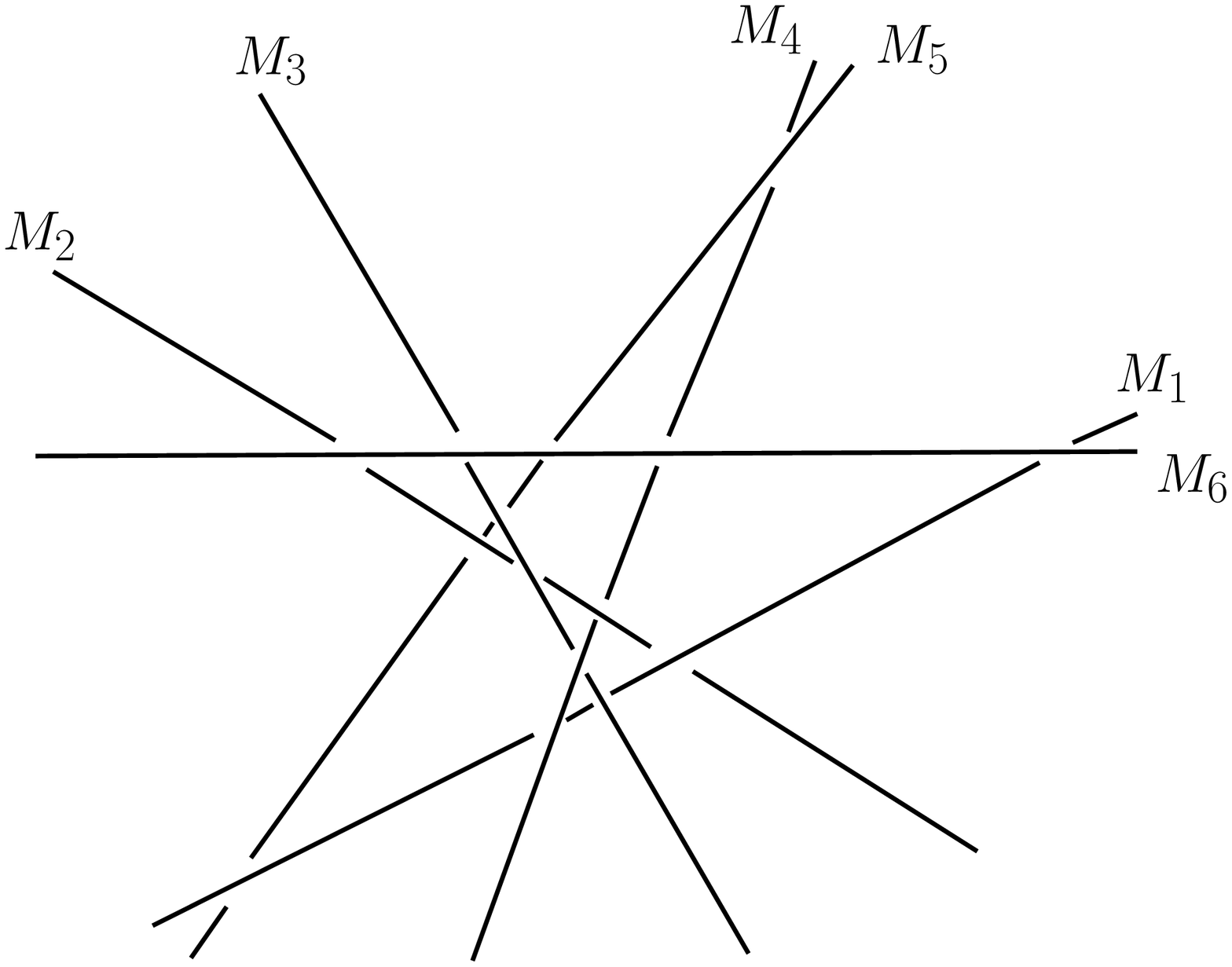}}}\hspace{0.4mm}
\subfigure[$\mathbf{L}$]{\scalebox{0.24}{\includegraphics{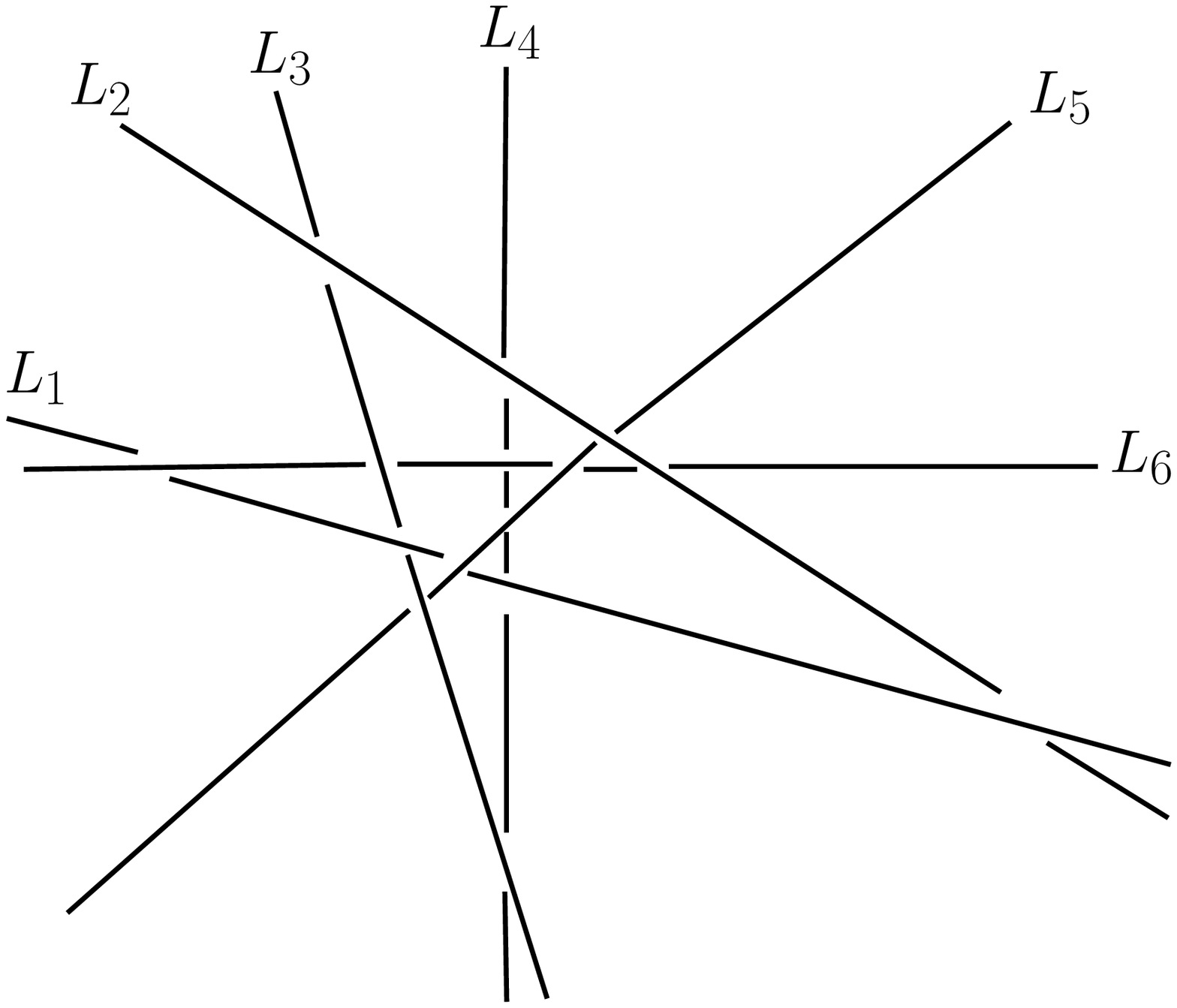}}}
\caption{A join $J(123654)$ and configurations $\mathbf{M}$ and $\mathbf{L}$}
\label{ML}
\vskip-3mm\end{figure}

\begin{theorem}\label{maintheorem}
A real Schl\"afli six $\L$ can be connected by a deformation with the configuration
\begin{enumerate}
\item
$J(123456)$ or  its mirror partner $J(654321)$ if $\L$ is hexagonal,
 \item
$J(123654)$ if $\L$ is bipartite,
 \item
$J(214365)$ or its mirror partner $J(563412)$ if $\L$ is tripartite,
  \item
$\CM$ (shown on Figure \ref{ML}) or its mirror partner $\mathbf{\bar{M}}$ if $\L$ is icosahedral.
  \end{enumerate}
\end{theorem}

\subsection{Scheme of proof}\label{scheme-of-proof}
First, we will observe that Schl\"afli sixes $\L=\{L_1,\dots,L_6\}$ have the following property of  {\it homogeneity}:
all six configurations $\L_i=\L\sm\{L_i\}$, $i=1,\dots,6$, obtained by
dropping a line from $\L$ belong to one coarse deformation class (see definition in Subsection \ref{delnotation})
and this class is different for different kinds of Schl\"afli sixes.
Next, we find the deformation classes in the Mazurovski\u\i~ list having this property and forbid among these classes,
class $\la\mathbf{L}\ra$.
The remaining four homogeneous deformation classes are the ones that appear in Theorem \ref{maintheorem}.

\section{Preliminaries}\label{preliminaries}

\subsection{Deformation and coarse deformation classes}\label{delnotation}
Let us denote by $\CCC_n$ the space formed by configurations $\L=\{L_1,\dots,L_n\}$ of n skew lines in $\Rp3$,
and by $[\CCC_n]$
 the set of connected components  of $\CCC_n$, which can be viewed also as
 {\it deformation equivalence classes}, where the equivalence means existence of a path (deformation)
connecting configurations in $\CCC_n$.
For $\L\in\CCC_n$ let $[\L]\in[\CCC_n]$ denote its deformation class. A mirror partner $\bar\L$ of $\L\in\CCC_n$ as it was defined in Subsection \ref{results}
is not unique, but its class $[\bar\L]\in\CCC_n$ is well-defined, depends only on $[\L]$ and will be called the {\it mirror partner class} for $[\L]$.
If $[\L]\ne[\bar\L]$, then
$\L$ is called {\it chiral configuration} and its class $[\L]$ {\it chiral class}.
Otherwise (in the case of $[\L]=[\bar\L]$) configuration $\L$ and its class $[\L]$ are called
{\it achiral}. The union $\la\L\ra=[\L]\cup[\bar\L]$ will be called {\it coarse deformation class} of $\L$,
and the set of coarse deformation classes in $\CCC_n$ will be denoted by $\la\CCC_n\ra$.

The set of real Schl\"afli sixes will be denoted by $\Sch$, its set of connected components
(deformation equivalence classes) by $[\Sch]$, and the set of coarse deformation classes by $\la\Sch\ra$
(mirror partners and chirality/achirality are defined similarly).

\begin{theorem}[{\bf Segre}]\label{Segretheorem}
There exist four coarse deformation classes in $\la\Sch\ra$
formed respectively by hexagonal, bipartite, tripartite and icosahedral Schl\"afli sixes.

Bipartite Schl\"afli sixes are achiral and the others are chiral. Thus, there exist 7 deformation classes in $[\Sch]$:
three pairs of mirror partners for hexagonal, tripartite and icosahedral Schl\"afli sixes and one class formed by bipartite Schl\"afli sixes.
\qed\end{theorem}

\begin{remark}
The work of Segre \cite{Se} does not contain in fact an explicit classification of real Schl\"afli sixes:
it contains instead a deformation classification of real double sixes in \textsection 56 (existence of four deformation classes).
Along with his description of the monodromy (permutations of their lines under deformations), it implies that
a pair of Schl\"afli sixes forming a double six must be mirror partners (and in particular, must be of the same kind),
which immediately implies Theorem \ref{Segretheorem}.
\end{remark}

\subsection{Join and spindle configurations}
Returning to the definition of $J(\s)$ for $\s\in S_n$,
note that neither a cyclic shift of numeration of the points $p_i\in L^p$,
nor such a shift for points $q_i\in L^q$ may change the class $[J(\s)]$.
In other words, $[J(c_1\s c_2)]=[J(\s)]$, where $c_1,c_2\in S_n$ are cyclic permutations. So, $J(\s)$
depends only on the orbit denoted $[\s]$ of the corresponding action of $\Z_n\times \Z_n$ on $S_n$.
By analogy with the Segre pentagrams, we can associate to $\s\in S_n$ a {\it permutation diagram}
\begin{figure}[h]
\centering
\subfigure[]{\scalebox{0.4}{\includegraphics{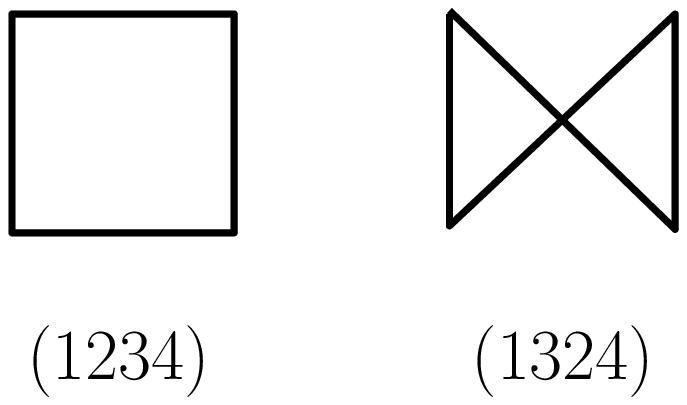}}}\hspace{1cm}
\subfigure[]{\scalebox{0.4}{\includegraphics{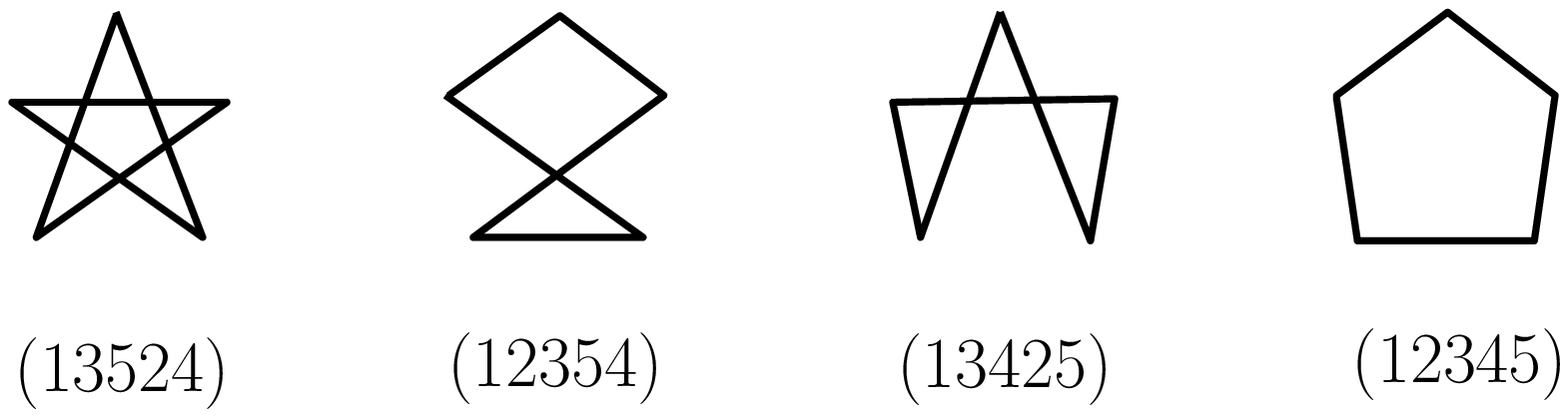}}}\\
\vskip-2mm
\caption{Quadrilateral and pentagonal coarse diagrams (pentagrams).}
\vskip-3mm
\label{ngon}
\end{figure}
formed by the vertices of a regular $n$-gon which are consecutively numerated by $1,2,\dots, n$ in counterclockwise direction and
are connected by a broken line $\s(1),\s(2),\dots,\s(n),\s(1)$.
Passing from $\s$ to its class $[\s]$ means that we forget the numeration of the vertices and consider such diagrams equivalent if they differ by
a rotation.
The coarse deformation class $\la J(\s)\ra$ can be similarly encoded by a {\it coarse diagram},
in which we forget also the direction of edges and which we do not distinguish from its reflection across a line,
see examples for $n=4,5$  on Figure ~\ref{ngon} and for $n=6$ on Figure \ref{hexagrams}.

We say that a configuration $\L=\{L_1,\dots,L_n\}$ of skew lines is a {\it spindle configurations}, if its lines have one common transversal called the {\it axis} of $\L$.
Such  a configuration is deformation equivalent to a join configuration $J(\s)$ where $\s\in S_n$ can be found like in the definition of Segre for $n=5$.
Namely, let $\pi_i$ denote the plane spanned by $L_i$ and the axis, $L^s$, of a spindle configuration $\L$.
Let us choose the numeration of lines $L_i$ so that $\pi_i$ appear consecutively in the pencil of planes rotating around axis $L^s$
in the negative direction with respect to some chosen orientation of $L^s$.
Then points $s_i=L^s\cap L_i$, $i=1,\dots, n$, follow on line $L^s$ in some order $s_{\s(1)},\dots,s_{\s(n)}$ in the positive direction of $L^s$
which defines a permutation $\s\in S_n$ (up to cyclic shifts, like in the case of join configurations).

\begin{proposition}\label{spindle-to-join}
If the above construction associates to a spindle configuration $\L=\{L_1,\dots,L_n\}$ a permutation $\s\in S_n$,
then $[\L]=[J(\s)]$.
\end{proposition}

\begin{proof}
Note that rotation of each line $L_i$ around point $s_i$ in the plane $\pi_i$
does not lead to intersection with the other lines of $\L$ as long as we rotate apart from the axis $L^s$.
So, we can choose any line $L^p$ disjoint from $L^s$, and by such rotation of each $L_i$ make it pass
through the point $p_i=L^p\cap\pi_i$.
 The joint configuration that we obtain realizes clearly the same permutation $\s$ as $\L$.
\end{proof}
\begin{figure}[h]
\vskip-2mm\centering
\subfigure[Chiral]{\scalebox{0.28}{\includegraphics{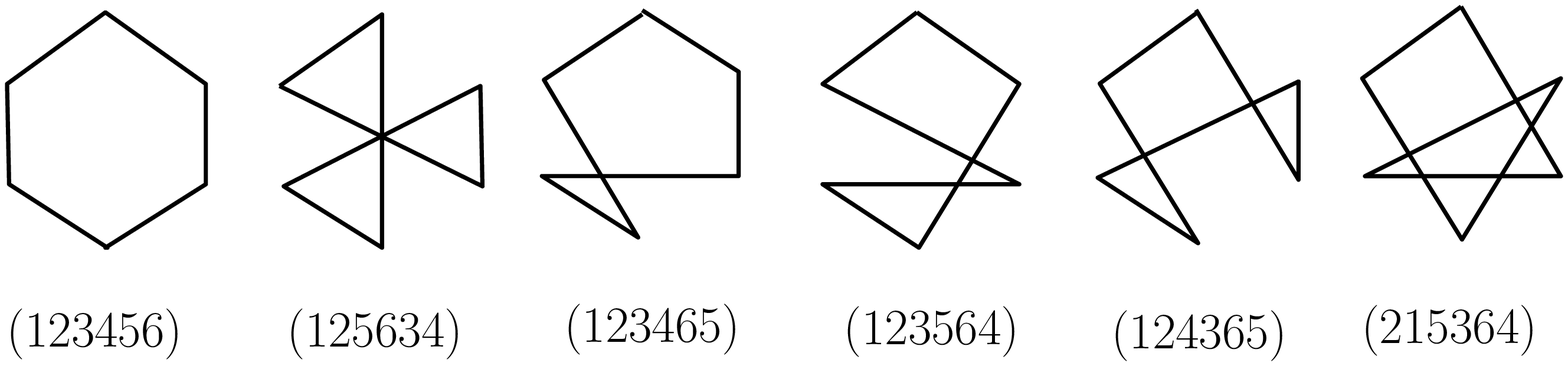}}}\hspace{1cm}
\subfigure[Achiral]{\scalebox{0.28}{\includegraphics{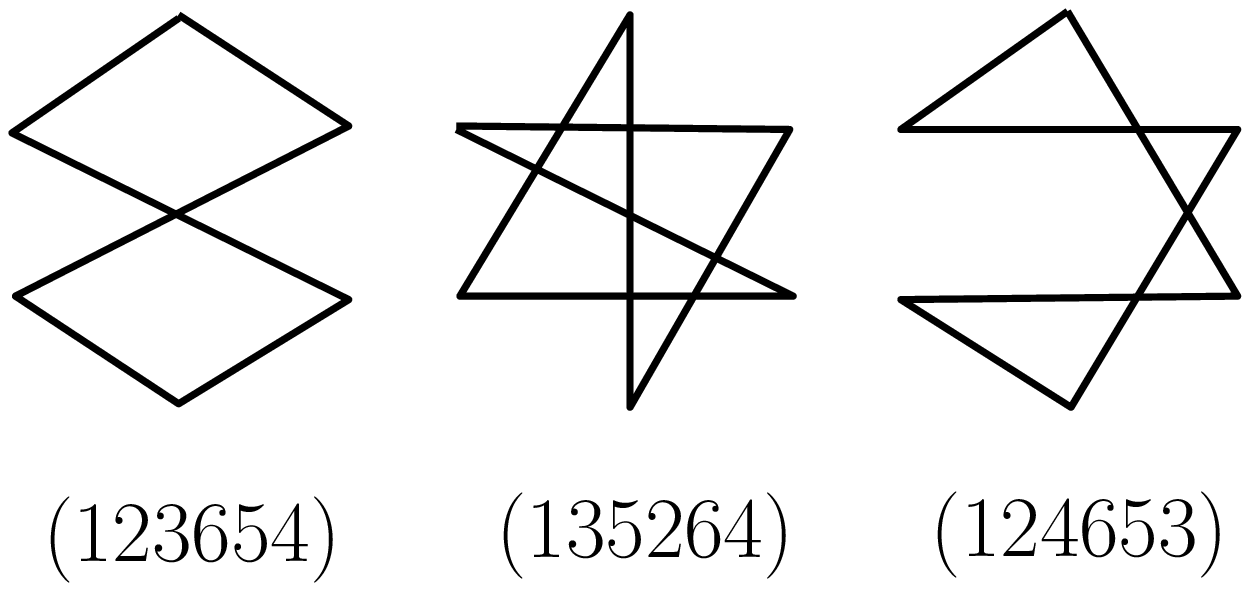}}}
\vskip-1mm\caption{Hexagonal coarse diagrams of join configurations of 6 lines}
\label{hexagrams}
\vskip-5mm\end{figure}

\subsection{Viro-Mazurovski\u\i~ classification of configurations $\L\in\CCC_n$, $n\le6$}\label{Viro-Masurovski}
The cases of $n\le3$ are quite trivial: space $\CCC_2$ is obviously connected and  $\CCC_3$ has two connected components: $[J(123)]$ and $[J(321)]$ (mirror partners).
The cases of $n=4,5$ were analyzed by O.~Viro (see \cite{DV}).
There are 3 deformation classes in $\CCC_4$: a pair of mirror partners $[J(1234)]$ and $[J(4321)]$,
and an achiral deformation class $[J(1243)]$.
In $\CCC_5$ there are 7 deformation classes and among them 3 pairs of mirror partners:
 $[J(12345)]$ with $[J(54321)]$, $[J(12354)]$ with $[J(45321)]$, $[J(13425)]$ with $[J(52431)]$,
and one achiral class $[J(13524)]$.

For $n=6$ a deformation classification was obtained by V.~Mazurovski\u\i~ \cite{Maz}.
\begin{theorem}[{\bf Mazurovski\u\i}]\label{Mazurovski}
Space $\CCC_6$ has $19$ deformation classes. Among them 4
do not contain join configurations, namely, classes $[M]$, $[L]$
of the configurations $M$ and $L$ shown on Figure ~\ref{ML}, and their mirror partners $[\bar M]$ and $[\bar L]$.
The remaining 15 classes include 6 mirror partners, forming together 6 coarse classes
$\la J(123456)\ra$,  $\la J(214365)\ra$,   $\la J(123465)\ra$,
$\la J(123564)\ra$, $\la J(124365)\ra$,  $\la J(215364)\ra$,
and 3 achiral deformation classes
$[J(123654)]$, $[J(135264)]$,  and $[J(124653)]$. So, it gives 11 coarse deformation classes.
\end{theorem}

\subsection{Triple linking indices}\label{tln}
In this subsection we
denote by $\overrightarrow{L}$ a line $L$ equipped with some orientation.
Recall that the linking number of a pair of skew lines $\overrightarrow{L_{1}},\overrightarrow{L_{2}}$ in $\Rp3$
(that has a canonical orientation)
is $\pm\frac{1}{2}$. To avoid fractions, we consider its double, $\lk(\overrightarrow{L_{1}},\overrightarrow{L_{2}})\in\{\pm1\}$ and call it the \textit{linking index} of
$\overrightarrow{L_{1}}$ and  $\overrightarrow{L_{2}}$.
Note that reversing the orientation of $\overrightarrow{L_1}$ or $\overrightarrow{L_2}$ leads to alternation of this index,
which implies that the following {\it triple linking index} introduced by Viro \cite{DV} for skew lines $L_{1}$, $L_{2}$, and $L_{3}$
$$\lk(L_{1},L_{2},L_{3})=\lk(\overrightarrow{L_{1}},\overrightarrow{L_{2}})\cdot
\;\lk(\overrightarrow{L_{1}},\overrightarrow{L_{3}})\cdot
\;\lk(\overrightarrow{L_{2}},\overrightarrow{L_{3}})$$
is independent of the orientations of lines $L_{i}$ and of their order (cf. \cite{DV}).

\begin{lemma}\label{permutation-linking}
Assume that  $J(\s)=\{L_1,\dots,L_n\}$ is a join configuration, $\s\in S_n$. Then
for any $1\le i<j<k\le n$, we have
$$\lk(L_i,L_j,L_k)=(-1)^{r_{\s(i)\s(j)\s(k)}}$$
where $r_{abc}$ is the number of inversions of order in the sequence $a,b,c$.
\end{lemma}
\begin{proof}
Our convention on the orientations of $L^p$ and $L^q$ implies that
$\lk(L_i,L_j)$ is $1$ if $\s(i)<\s(j)$ and $-1$ otherwise.
\end{proof}

For $\L\in\CCC_n$ we can count the numbers $n_\pm(\L)$ of triples $\{L_1,L_2,L_3\}\subset\L$ with $\lk(L_{1},L_{2},L_{3})=\pm1$
and call the difference $\sign(\L)=n_{+}(\L)-n_{-}(\L)$
the \textit{signature} of $\L$.
It is clearly a deformation invariant of configurations $\L\in\CCC_n$, and $\sign(\bar \L)=-\sign(\L)$ for a mirror partner $\bar \L$ or $\L$.
For $n=5$ this invariant turns out to be complete: it is different for different deformation classes $[\L]\in[\CCC_5]$.
Namely, for $J(\s)$ where $\s$ is $(12345)$, $(12354)$, $(13425)$, and $(13524)$ the signature is respectively $10$, $4$, $2$, and $0$,
and for $\bar J(\s)$ the signature is opposite.

\subsection{Internal adjacency and reducibility}\label{adjacency}
Following \cite{DV}, we call a pair of lines $L_i$, $L_j$ of $\L=\{L_1,\dots,L_n\}\in\CCC_n$ \textit{internally adjacent} in $\L$ if
there exists a continuous family $\L(t)=\{L_1(t),\dots,L_n(t)\}$, $t\in[0,1]$, of configurations
such that $\L(0)=\L$, $\L(t)\in\CCC_n$ for $t\in[0,1)$, and the limit configuration $\L(1)$ has $n-1$ skew lines,
one of which is double line $L_i(1)=L_j(1)$ (result of merging).
A configuration $\L$ is called {\it reducible}
if it contains an internally adjacent pair of lines, and {\it irreducible} otherwise.

One can easily show that a join configuration $J(\s)$
is irreducible if and only if $|\s(i+1)-\s(i)|>1$ for $i=1,\dots, n$ (here, $\s(n+1)$ is $\s(1)$),
which means that a permutation diagram of $\s$ does not contain a side of the $n$-gon.
In particular, from Figure \ref{hexagrams} one can conclude that irreducible join configurations in $\CCC_6$ form only one deformation class $\la J(135264)\ra$
(cf. \cite{Maz}).
It was also observed in \cite{Maz} that the configurations $\CL$, $\CM$ and their mirror partners are irreducible too.

\subsection{Signature spectrum}\label{dlc}
By the {\it signature spectrum} of a configuration  $\L=\{L_1,\dots,L_n\}\in\CCC_n$, we mean the list of signatures of its subconfigurations
$\L\sm\{L_i\}$, $i=1,\dots,n$: each value of signature appears with a multiplicity indicated how many times it occurs, so that
the sum of all multiplicities is $n$.
Clearly, the signature spectrum is a deformation invariant of $\L$ that refines $\sign(\L)$ and can be easily calculated for join configurations using Lemma \ref{permutation-linking}.
As an invariant of coarse deformation equivalence, we can use {\it coarse signature spectrum} in which instead of signatures of subconfigurations $\L\sm\{L_i\}$ we list only their absolute values.
The table on Figure \ref{alldeformationclasses} contains the results of calculation of such spectra for 6-configurations:
the last column contains {\it pentagram spectra}, in which pentagrams of 5-subconfigurations replace the corresponding signatures (see Subsection \ref{tln}).

\begin{figure}[ht]
\centering
{\scalebox{0.33}{\includegraphics{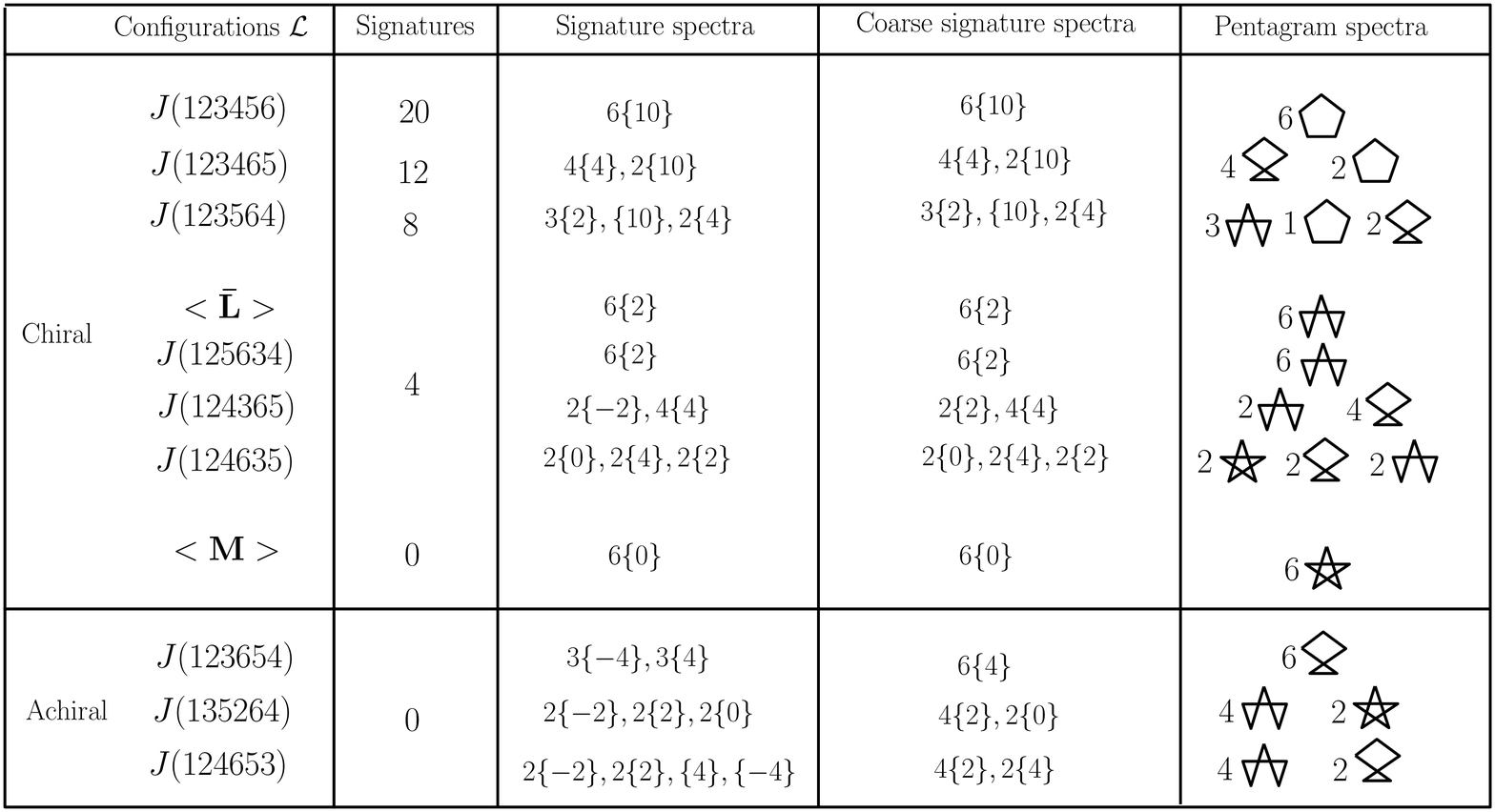}}}
\caption{Deformation invariants of configurations $\L\in\CCC_6$}
\label{alldeformationclasses}
\vskip-3mm\end{figure}

\section{Proof of Theorem~\ref{maintheorem}}\label{cubicsurface}

\subsection{Homogeneity of Schl\"afli sixes}\label{homogeneity}
{\it Homogeneity} of $\L$ means the following.
\begin{proposition}\label{homogeniety-proposition}
For any real Schl\"afli six $\L=\{L_1,\dots,L_6\}$,
all 6 subconfigurations  $\L\sm\{L_i\}$, $i=1,\dots,6$,
are coarse deformation equivalent.
\end{proposition}

\begin{proof}
Each of the subconfigurations $\L\sm\{L_i\}$ is spindle with the axis $L_i'$ from the complementary Schl\"afli six $\L'=\{L_1',\dots,L_6'\}$.
As we mentioned in Subsection \ref{Segre-classification},
Segre proved that the associated permutations are equivalent for every $i=1,\dots,6$:
\begin{theorem}[{\bf Segre}]\label{homogeinety-theorem}
The six spindle subconfigurations $\L\sm\{L_i\}$
have the same associated coarse pentagrams that depend on the kind of  Schl\"afli six $\L$ as it is indicated on Figure \ref{ACC6}.
\qed\end{theorem}

Now it is left to apply
Proposition \ref{spindle-to-join} and to use the deformation classification of 5-line configurations by Viro.
\end{proof}


\subsection{Reducibility of Schl\"afli sixes}\label{reducibility}
For a Schl\"afli six $\L=\{L_1,\dots,L_6\}$ on a cubic surface $X$ we associated in Subsection \ref{Segre-classification}
a planar configuration $\P=\{p_1,\dots,p_6\}\subset \mathbb{P}^2$ obtained by contraction lines $L_i$ on $X$ to point $p_i$.
This construction is invertible: by blowing up plane $\mathbb{P}^2$ we obtain
a del Pezzo surface of degree 3 that can be embedded to $\mathbb{P}^3$ by anti-canonical linear system and
realized there as a cubic surface; the exceptional divisors of blowing up at $p_i$ become lines $L_i$.
This gives an isomorphism of the moduli spaces formed by projective classes of pairs $(X,\L)$, where cubic $X$ is non-singular,
and projective classes of 6-point configurations $\P\subset \mathbb{P}^2$, which are generic (contains no collinear triples of points
and all 6 points are not coconic).
GIT gives an extension of this isomorphism to cubic surfaces with a node (see \cite{D}, \textsection{8-9}, particularly Remark 9.4.1):
the corresponding configurations $\Cal P$ have either (a) a collinear triple, or (b) all 6 points coconic, or (c) a pair of merged points (in the sense of Hilbert schemes).
In the cases (a) and (b) the corresponding lines on $X_0$ remain distinct, but the corresponding triple in case (a) or all six in case (b) pass through the node.
In the case (c), the pair of lines that represent merging points should merge on the nodal cubic,
while the other 4 lines do not intersect neither each other nor the double line.
\begin{lemma}\label{reducibility}
Bipartite and tripartite Schl\"afli sixes are reducible configurations.
\end{lemma}
\begin{proof}
Any given real Schl\"afli six $\L=\{L_1,\dots,L_6\}$ realized on some real cubic surface $X$ is represented by
a configurations of six points in general position, $\P=\{p_1,\dots,p_6\}\subset\Bbb P^2$ as explained above.
If $\L$ is not icosahedral, then the inseparability graph of $\P$ has some edge, $[p_ip_j]$ (see Figure \ref{adjgraph6}).
Let $p_i(t)$, $t\in[0,1]$, be a parametrization of this edge, so that $p_i(0)=p_i$, $p_i(1)=p_j$, and
$\P(t)\subset\Bbb P^2$ denote the configuration formed by the five {\it fixed points} of $\P\sm\{p_i\}$ and one {\it moving point} $p_i(t)$,
so that $\P(0)=\P$ and $\P(1)$ has a double point $p_j$.
The correspondence between pairs $(X,\L)$ and configurations $\P$ extended to $\P(t)$ yields a continuous family $(X(t),\L(t))$, $t\in[0,1]$,
$(X(0),\L(0))=(X,\L)$. Note that under our assumption, points of $\P(t)$ are in general position for any
$t\in[0,1)$. Namely, point $p_i(t)$ is not collinear with any pair of other points of $\P(t)$ by definition of the inseparability graph.
Also, all six points of $\P(t)$ cannot be coconic, unless configuration $\P$ is hexagonal (and this case is excluded from the Lemma).
To show the latter fact, note that for a conic passing through the 5 points of $\P\sm\{p_i\}$ and intersecting edge $[p_ip_j]$
at an interior point, $p_i'$, (in addition to point $p_j$), the configuration $(\P\sm\{p_i\})\cup\{p_i'\}$ is hexagonal since all the points lie on a non-singular conic,
and replacing $p_i'$ by $p_i$ keeps the configuration hexagonal, since $[p_ip_j]$ is an edge of the inseparability graph
(see Figure \ref{cross-conic}:
\begin{figure}[h]
\centering
{\scalebox{0.4}{\includegraphics{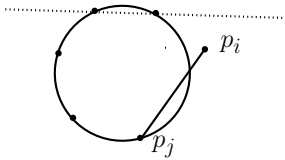}}}
\caption{Edge $[p_ip_j]$ that crosses the conic passing through the points of $\P\sm\{p_i\}$
cannot cross the dotted line, thus, $\P$ is hexagonal}
\label{cross-conic}
\vskip-1mm\end{figure}
the dotted line cannot cross the edge $[p_ip_j]$).

Thus, cubic surfaces $X(t)$ remain non-singular for $t\in[0,1)$
and configurations $\L(t)\subset X(t)$ remain Schl\"afli sixes.
These surfaces degenerate to a nodal cubic surface $X_1$, where $\L(1)$ contains a double line obtained by merging
of the lines $L_i$ and $L_j$ (corresponding to merging of points $p_i$ and $p_j$).
The other 4 lines do not intersect each other and the double line that we obtained.
 \end{proof}

\begin{remark}
Hexagonal Schl\"afli sixes are also reducible, while
icosahedral ones are not. Both facts follow from Theorem \ref{maintheorem}
(but not required for its proof).
\end{remark}

\subsection{Completion of the proof of Theorem \ref{maintheorem}}\label{summary}
The homogeneity property of Schl\"afli sixes established in Proposition \ref{homogeniety-proposition}
implies that their coarse spectra contain one value repeated 6 times.
There exist precisely 5 such coarse
deformation classes in the Mazurovski\u\i~ list in Theorem \ref{Mazurovski} (see Figure \ref{alldeformationclasses}).
For hexagonal, bipartite and icosahedral Schl\"afli sixes we have a unique homogeneous class from Mazurovski\u\i~ list with the same signature spectrum.
 For tripartite Schl\"afli sixes we have two candidates, classes $\la L\ra$ and $\la J(214365)\ra$, that have the same derivative signature list.
However, configuration $L$ (as thus, its mirror partner $\bar L$) is irreducible, as it was established in \cite{Maz}, so, using Lemma \ref{reducibility}
we may conclude that tripartite Schl\"afli sixes must represent coarse deformation class $\la J(214365)\ra$.
\qed

\section{Concluding remarks}

\subsection{An alternative proof of Segre Theorem \ref{homogeinety-theorem}}
With a configuration $\L=\{L_1,\dots,L_6\}\in \CCC_6$ we associate a {\it chirality graph} $\G_\L$, whose vertex set is $\L$, and which includes edge
$[L_iL_j]$, $1\le i<j\le 6$, if and only if the 4-line configuration $\L\sm\{L_i,L_j\}$ is chiral.
Let $\deg(L_i)$ denote the degree of vertex $L_i$ in graph $\G_\L$.
\begin{figure}[ht]\vskip-2mm
 \subfigure{\scalebox{0.45}{\includegraphics{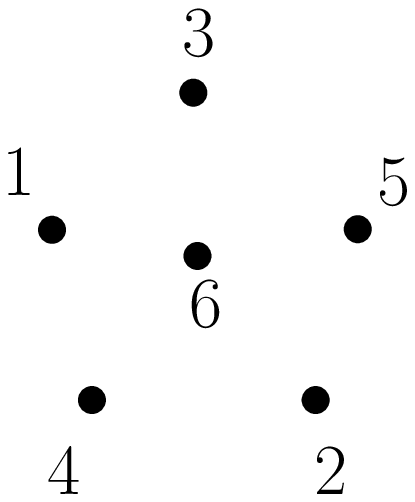}}}\hspace{0.7cm}
 \subfigure{\scalebox{0.45}{\includegraphics{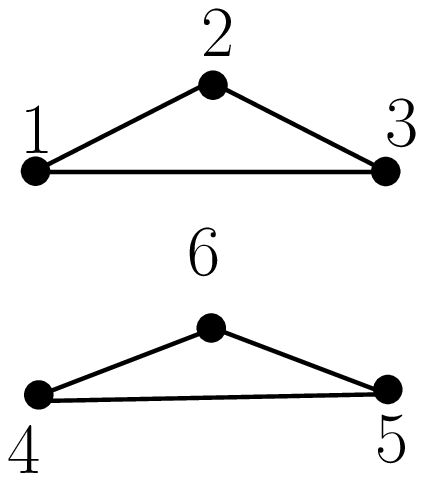}}}\hspace{0.7cm}
\subfigure{\scalebox{0.45}{\includegraphics{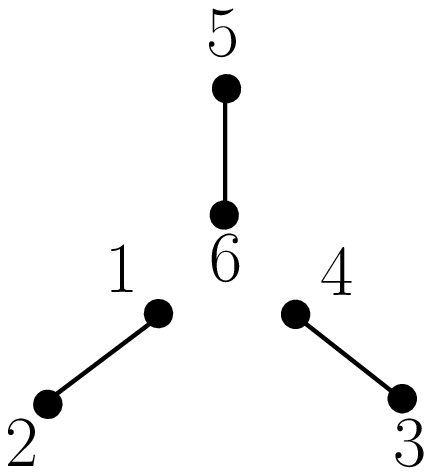}}}\hspace{0.7cm}
\subfigure{\scalebox{0.45}{\includegraphics{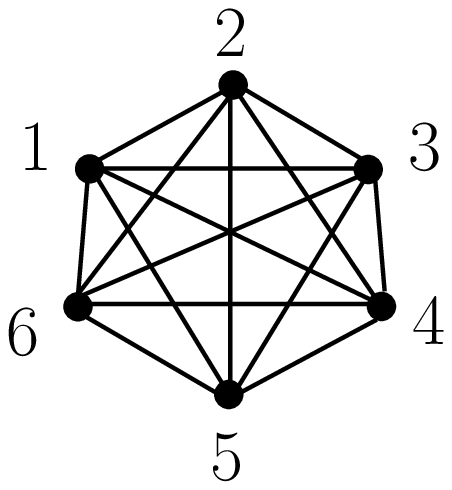}}}
\caption{Chirality graphs for the 4 kinds of real Schl\"afli sixes}
 \label{chiralgraph6}
\vskip-4mm\end{figure}
\begin{lemma}
$|\sign(\L\sm\{L_i\})|=2\deg(L_i)$.
\end{lemma}

\begin{proof}
It is sufficient to verify  that the absolute value $|\sign(\L')|$ for any
configuration $\L'=\{L_1,\dots,L_5\}\in \CCC_5$
is twice the number of 4-line chiral subconfigurations of $\L'$.
Due to Viro classification (see Subsection \ref{Viro-Masurovski})
it is enough to justify this claim for $J(12345)$ (five chiral 4-subconfigurations, signature $\sign=10$),
$J(12354)$ (two such subconfigurations and $\sign=4$),
$J(13425)$ (one such subconfiguration and $\sign=2$)
and $J(13524)$ (no such subconfigurations and $\sign=0$).
\end{proof}
It follows that $\L\in \CCC_6$ is a homogeneous configuration if and only if $\G_\L$ is a regular graph (the degrees of all vertices are the same).

For real Schl\"afli sixes $\L$ it is easy to differ chiral 4-subconfigurations from achiral ones using the corresponding 6-point configurations $\P\subset\Bbb P^2$ on Figure \ref{adjgraph6}.
The two transversals for a 4-subconfiguration $\L\sm\{L_i,L_j\}$ are the lines $L_i'$ and $L_j'$
of the complementary Schl\"afli six $\L'$ that are represented in $\Bbb P^2$ by two conics
passing through $\P\sm\{p_i\}$ and $\P\sm\{p_j\}$. Their mutual position determines chirality of
 $\L\sm\{L_i,L_j\}$ as it is shown of Figure \ref{chirality-via-conics}.
\begin{figure}[h]
\vskip-1mm \centering
    \subfigure[(a) Chiral subconfiguration]{\scalebox{0.4}{\includegraphics{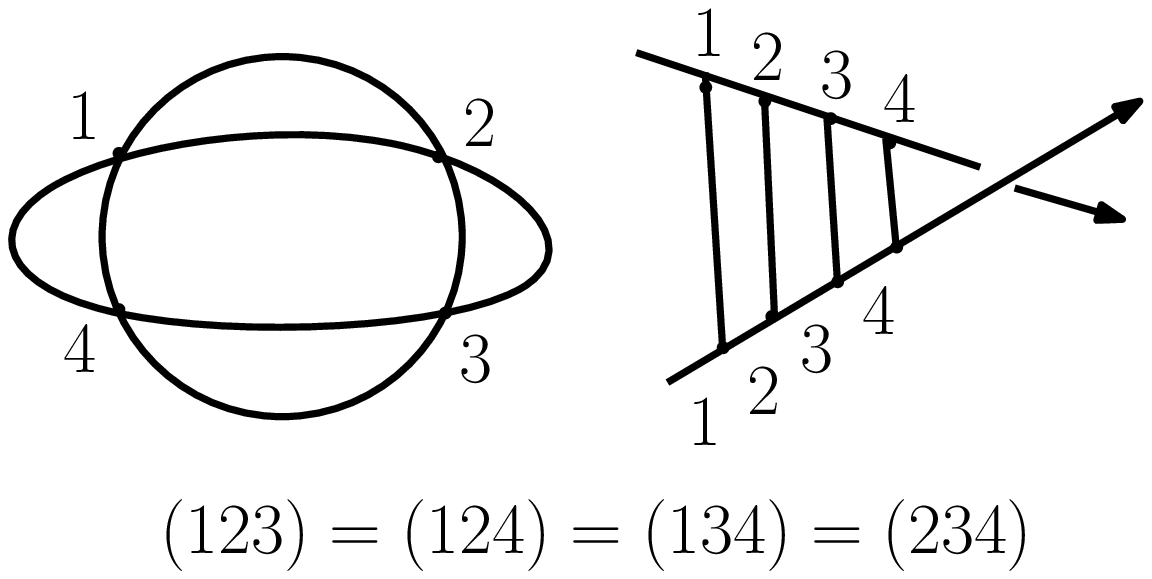}}}\hspace{2.2cm}
    \subfigure[(b) Achiral subconfiguration]{\scalebox{0.4}{\includegraphics{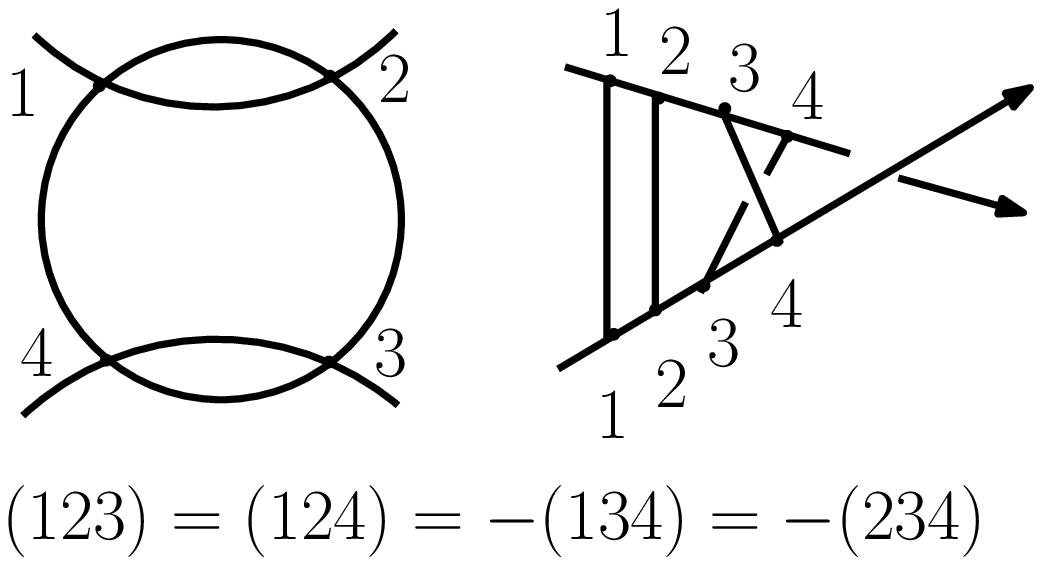}}}
\caption{The chirality of quadratuples of lines.}
 \label{chirality-via-conics}
\vskip-2mm\end{figure}
This allows easily to construct the chirality graph $\G_\L$ for $\L$ (see Figure~\ref{chiralgraph6})
and to justify Theorem  \ref{homogeinety-theorem}.
Namely, graph $\G_\L$ is spanned by the inseparability graph shown on Figure \ref{adjgraph6}.

\subsection{Cubic surfaces of type $F_2$}
One may consider configurations of skew real lines on real non-singular cubic surfaces $X$ of types other than $F_1$.
In the case of type $F_2$, among the 15 real lines on $X$ one can find maximum 5 disjoint ones.
By blowing down them in $X$ we obtain 5 generic points on an ellipsoid (such that no 4 points are coplanar).
As it was observed by Segre \cite{Se}, such 5-point configurations form one real deformation class.
Thus, configurations of 5 real skew lines on $X$ of type $F_2$
belong to the same coarse deformation class in the Viro list (see Subsection \ref{Viro-Masurovski}).
Arguments similar to the ones in Lemma \ref{reducibility} imply that this class is $\la J(12345)\ra$.

For cubics of other types our question is trivial: if $X$ has types
$F_3$, $F_4$ and $F_5$ then the maximum number of disjoint real lines on $X$ is respectively $3$, $1$ and $1$.

\subsection{The Viro 2-cocycle}
For any set 4 skew lines, $L_i,L_j,L_k, L_m$, we have
$$ \lk(L_i,L_j,L_k)\lk(L_i,L_j,L_m)\lk(L_i,L_k,L_m)\lk(L_j,L_k,L_m)=1
$$
(see \cite{DV}). This is just a cocycle relation for $\lk(L_i,L_j,L_k)$ that is viewed as a $\Z/2$-valued simplicial 2-cochain
of the simplex $\D_\L$ with the vertex set $\L=\{L_1,\dots,L_n\}$. Here $\L$ is any configuration of skew lines in $\Rp3$,
and the coefficients group $\Z/2$ appears in multiplicative form, as group $\{1,-1\}$.

We may consider also a pair of complementary to each other $\Z/2$-valued 2-chains. One is the sum of 2-simplices $[L_iL_jL_k]$
such that $\lk(L_i,L_j,L_k)=1$,
and the other is the sum of such 2-simplices with $\lk(L_i,L_j,L_k)=-1$.
In general they need not to be cycles,
but, clearly, for $n=6$, if one of them is a cycle, then the other one is a cycle too.
Moreover, it turns out that for $n=6$ these 2-chains
{\it are 2-cycles if and only if $\L$ is homogeneous}, and in particular, it is so in the case of Schl\"afli sixes.

\subsection{Icosahedral pair of 2-cycles}
In the case of icosahedral Schl\"afli six $\L=\{L_1,\dots,L_6\}$, the union of the triangles of each of the above
2-cycles is a subset of $\D_\L$ homeomorphic to $\Rp2$ and triangulated like the quotient of an icosahedron by the antipodal map.
If lines $L_i$ are numerated like the points $p_i$ on Figure \ref{ACC6}, then
one of these 2-cycles includes triangles  $[L_iL_jL_k]$
with the following triples $(ijk):$
$$ (123), (234), (345), (451), (512), (136), (356), (526), (246), (416),
$$
and the other 2-cycle includes the remaining 10 triangles with the triples $(ijk):$
$$(124), (235), (341), (452), (513), (126), (236), (346), (456),(516).
$$


\end{document}